\documentclass[11pt,oneside]{article}

\usepackage[round,colon,authoryear]{natbib} 

\usepackage{graphicx}
\usepackage{amsfonts}
\usepackage{amssymb}
\usepackage{amsmath}
\usepackage{amsthm}
\usepackage{enumerate}
\usepackage{times}
\usepackage{eurosym}
\usepackage{url}
\usepackage{bm}
\usepackage{comment}
\usepackage{hyperref}
\usepackage{enumitem}

\newtheorem{theorem}{Theorem}

\newtheorem{lemma}{Lemma}

\theoremstyle{remark}

\theoremstyle{remark}
\newtheorem{ex}{Example}
\theoremstyle{remark}

\theoremstyle{theorem}

\newcommand{\E}{{\mathbb{E}}}

\renewcommand{\P}{{\mathbf{P}}}

\newcommand{\1}{\mathbf{1}}
\newcommand{\F}{\mathcal{F}}
\newcommand{\N}{\mathbb{N}}

\newcommand{\dd}{\mathrm{d}}
\newcommand{\lc}{[\![}
\begin{document}

\title{The Uniform Integrability of Martingales. On a Question by Alexander Cherny\thanks{%
I thank Zhenyu Cui for telling me about  the paper \citet{Cherny_2006} and  Mikhail Urusov for sharing his very helpful insights and his careful reading of an earlier version of this note. I am  indebted to Beatrice Acciaio, Ioannis Karatzas, Kostas Kardaras, Martin Larsson,  Nicolas Perkowski, and Pietro Siorpaes for many discussions on the subject matter of this note.   I am grateful to an anonymous referee and an associate editor for their detailed and helpful comments. 
I  acknowledge generous support from the Oxford-Man Institute of Quantitative Finance, University of Oxford,  
}}
\author{Johannes Ruf\thanks{%
E-Mail: j.ruf@ucl.ac.uk} \\
Department of Mathematics\\University College London}
\maketitle

\begin{abstract}
Let $X$ be a progressively measurable, almost surely right-continuous stochastic process such that $X_\tau \in L^1$ and $\E[X_\tau] = \E[X_0]$ for each finite stopping time $\tau$. In 2006, Cherny showed  that $X$ is then a uniformly integrable martingale provided that $X$ is additionally nonnegative. Cherny then  posed the question whether this implication also holds even if $X$ is not necessarily nonnegative. We provide an example that illustrates that this implication is wrong, in general. If, however, an additional integrability assumption is made on the limit inferior of $|X|$ then the implication holds. Finally, we argue  that this integrability assumption  holds if the stopping times are allowed to be randomized in a suitable sense.


\end{abstract}



\section{Introduction}

We fix a filtered probability space $(\Omega, \F, \mathbb F, \P)$ with expectation operator $\E[\cdot]$, where  $\mathbb F = (\F_t)_{t \in [0,\infty)}$ and $\F_\infty = \bigvee_{t \in [0,\infty)} \F_t \subset \F$. Furthermore, we fix a progressively measurable, almost surely right-continuous process $X$. We write $Z \in L^1$ if $\E[|Z|] < \infty$ for some random variable $Z$.  For some $\mathbb F$--adapted process $Y$ and some stopping time $\eta$ we write $Y^\eta$ to denote the process $Y$ stopped at time $\eta$; to wit, $Y^\eta_t = Y_{\eta \wedge t}$ for each $t \in [0,\infty)$.  All identifications and statements in the following are in the almost-sure sense.

We consider the following five statements:
\begin{enumerate}[label={\rm(\Roman{*})}, ref={\rm(\Roman{*})}]
	\item\label{I} $X$ is a uniformly integrable martingale.
	\item\label{II} $X_\infty = \lim_{t \uparrow \infty} X_t$ exists, $X_\tau \in L^1$ and $\E[X_\tau] = \E[X_0]$ for all stopping times $\tau$.
	\item\label{III}  $X_\tau \in L^1$ and $\E[X_\tau] = \E[X_0]$ for all finite stopping times $\tau$ and $\liminf_{t \uparrow \infty} |X_t| \in L^1$.
	\item\label{IV}  $X_\tau \in L^1$ and  $\E[X_\tau] = \E[X_0]$ for all finite stopping times $\tau$.
	\item\label{V}  $X_t \in L^1$ and  $\E[X_t] = \E[X_0]$ for all $t \in [0,\infty)$.
\end{enumerate}

The optional sampling theorem yields the implications \ref{I} $\Rightarrow$ \ref{II} $\Rightarrow$ \ref{III} $\Rightarrow$  \ref{IV}  $\Rightarrow$ \ref{V}.  The implication \ref{II} $\Rightarrow$ \ref{I} follows from the following simple argument. Fix $s,t \in [0,\infty]$ with $s <t$ and $A \in \F_s$. Let $\tau_1 = s$ and $\tau_2 = s \1_{A^c} + t \1_A$ denote two stopping times, where $A^c = \Omega \setminus A$. Then, by assumption, $\E[X_{\tau_1}] = \E[X_{\tau_2}]$, which yields   $\E[X_{s} \1_A] = \E[X_t \1_{A}]$. Thus we obtain the desired implication \ref{II} $\Rightarrow$  \ref{I}.  However, if only \ref{III} or \ref{IV} is assumed, this argument only yields the martingale property of $X$ but not its uniform integrability.

 \citet{Cherny_2006} now asks the question whether also the implication   \ref{IV} $\Rightarrow$  \ref{I} holds.  Before answering this question and discussing the role of \ref{III}, let us first briefly consider the statement in \ref{V}.   \citet{Hulley_2009} provides an example of a local martingale $X$, such that  \ref{V} is satisfied but $X$ is not a martingale.
Alternatively, if $X$ denotes Brownian motion started in $0$  then \ref{V} holds but \ref{IV} is not satisfied.  To see this, we only need to let $\tau$ denote the first hitting time of level $1$ by $X$. Thus, the implication \ref{V}  $\Rightarrow$ \ref{IV} does \emph{not} hold in general.

We now return to discuss the missing implications, namely whether \ref{III} or, more generally \ref{IV}, imply \ref{I} (or equivalently, \ref{II}).
 \citet{Cherny_2006} proves that these implications hold  if $X$ is nonnegative.  The following theorem proves that the implication \ref{III} $\Rightarrow$ \ref{I} holds always, not only if $X$ is nonnegative. 
  However, the example of the next section shows that \ref{IV} does not necessarily imply any of the statements \ref{I} -- \ref{III} if the nonnegativity assumption on $X$ is dropped.  Yet, as proven in Section~\ref{S:random}, if the filtered probability space  $(\Omega, \F, \mathbb F, \P)$ allows for some additional randomization, then these implications hold. More precisely, if the filtered probability space  $(\Omega, \F, \mathbb F, \P)$ allows for a $(0,1)$--uniformly distributed random variable,  measurable with respect to $\F_\eta$ for some finite stopping time $\eta$, then the implications  \ref{IV} $\Rightarrow$ \ref{I},   \ref{IV} $\Rightarrow$ \ref{II}, and   \ref{IV} $\Rightarrow$ \ref{III} hold.

\begin{theorem} \label{T:1}
	The statements \ref{I}, \ref{II}, and \ref{III} are equivalent. 
\end{theorem}
\begin{proof}
We only need to show the implication \ref{III} $\Rightarrow$ \ref{I}. We start by arguing that we may assume, without loss of generality, that $\mathbb F$ and $X(\omega)$ are right-continuous  for each $\omega \in \Omega$.  Towards this end, we set $\F_t^+ = \bigcap_{s >t} \F_s$ and $\mathbb F^+ = (\F_t^+)_{t \in [0,\infty)}$.  Next, we observe that the $\mathbb F$--martingale $X$ is also an $\mathbb F^+$--martingale due to Exercise~1.5.8 in \citet{SV_multi}. Now, Lemma~1.1 in \citet{F1972} yields the existence of a right-continuous version of $X$, which we call again $X$.  
Next, we fix a finite $\mathbb F^+$--stopping time $\widehat \sigma$ and set $\sigma = \widehat \sigma + 1$, which is a finite $\mathbb F$--stopping time by Theorem~IV.57 in \citet{Dellacherie/Meyer:1978}.  Then, $X^\sigma$ satisfies  \ref{II} and is therefore a uniformly integrable $\mathbb F^+$--martingale. The optional sampling theorem then also yields that $X_{\widehat \sigma} \in L^1$ and $\E[X_{\widehat \sigma}] = \E[X_0]$. Thus, \ref{III} also holds for all finite $\mathbb F^+$--stopping times, and we shall assume from now on, throughout this proof, that $\mathbb F$ and $X(\omega)$ are right-continuous  for each $\omega \in \Omega$.

We now construct
 a nondecreasing sequence $(T_n)_{n \in \N}$ of $[0,\infty)$--valued random variables such that $\lim_{n \uparrow \infty} |X_{T_n}| = \liminf_{t\uparrow \infty} |X_t| \in \F_\infty$.  For example, we can choose $T_n$ as the first time that $| |X| - \liminf_{t\uparrow \infty} |X_t|| \leq 1/n$.  Then $T_n$ is $\F_\infty$--measurable, due to the right-continuity of $X$, for each $n \in N$ (but not necessarily a stopping time).  Now we set $Y = \liminf_{n \uparrow \infty} X_{T_n} \in \F_\infty$ and note that $|Y| = \liminf_{t\uparrow \infty} |X_t|$, thus $Y \in L^1$ by assumption.

Next, let us consider the martingale $M$ given by $M_t = X_t - \E[Y|\F_t]$ for each $t \in [0,\infty)$, where we may use a right-continuous modification of the conditional expectation process thanks to $\mathbb F$ being right-continuous; see again Lemma~1.1 in \citet{F1972}.  Then \ref{III}  holds with $X$ replaced by $M$. We note that $\liminf_{t \uparrow \infty} |M_t| = 0$ since
$\liminf_{n \uparrow \infty} M_{T_n} = \liminf_{n \uparrow \infty} X_{T_n} - Y=0$, thanks to
 L\'evy's  martingale convergence theorem and the fact that $Y \in \F_\infty$.

 It is sufficient to show that the martingale $M$ is uniformly integrable, or, equivalently that $M \equiv 0$. Towards this end, we assume that there exists $\varepsilon \in (0,1)$ such that $\P(\sup_{t \in [0,1/\varepsilon)} M_{t}  > \varepsilon) > \varepsilon$.   We then let $\sigma_1$ be the first time that $M$ is greater than or equal to $\varepsilon$ and  $\sigma_2$  the first time after time $1/\varepsilon$ that $|M|$ is less than or equal to $\varepsilon^2/4$. Then $\sigma_2$ is finite since $\liminf |M_t| = 0$ and, with $\tau = \sigma_1 \wedge \sigma_2$, we may assume that $M_\tau \in L^1$ and obtain
\begin{align*}
	\E[M_\tau] = \E[M_{\sigma_1} \1_{\{\sigma_1 \leq \sigma_2\}}]  + \E[M_{\sigma_2} \1_{\{\sigma_1 > \sigma_2\}}] \geq 
		\varepsilon \P\left(\sigma_1 \leq \frac{1}{\varepsilon}\right) - \frac{\varepsilon^2}{4} \geq \frac{3 \varepsilon^2}{4} >  \frac{\varepsilon^2}{4} \geq  \E[|M_{\sigma_2}|] \geq \E[M_{\sigma_2}],
\end{align*}
which contradicts  \ref{III} with $X$ replaced by $M$.
Thus $M \leq 0$ and in the same manner, we can show that $M \geq 0$, which yields the statement.
\end{proof}

We briefly remark that if $X$ is nonnegative, then \ref{IV} yields that $X$ is a martingale, thus a nonnegative supermartingale and therefore $\liminf_{t \uparrow \infty} |X_t| \in L^1$.  Theorem~\ref{T:1} then yields Cherny's result, namely the implication \ref{IV} $\Rightarrow$ \ref{I} provided that $X$ is nonnegative.

\section{A counterexample}
We now construct a filtered probability space $(\Omega, \F, \mathbb F, \P)$ with a right-continuous martingale $X$ that satisfies  \ref{IV} and has a limit  $X_\infty = \lim_{t \uparrow \infty} X_t$, but is not uniformly integrable. 

\begin{ex} \label{Ex:2}
We let $\Omega = (\N  \cup \{\infty\}) \times \{-1,1\}$ and $\F$ the power set of $\Omega$. 
Next, we let $\P$ be the probability measure on $(\Omega, \F)$ such that $\P((n,i)) = 1/(4 n^2)$ for all $n \in \N$ and $i \in \{-1,1\}$ and  $\P((\infty,i)) = (1 - \pi^2/12)/2$ for all $i \in \{-1,1\}$. Since $\sum_{n \in \N} 1/(2n^2) = \pi^2/12 \in (0,1)$, this yields indeed a probability measure on $(\Omega, \F)$. 

We let $\sigma: \Omega \rightarrow [0,\infty], \, (\omega_1, \omega_2) \mapsto \omega_1$ denote the first component and $D: \Omega \rightarrow \{-1,1\}, \, (\omega_1, \omega_2) \mapsto \omega_2$  the second component of each scenario $(\omega_1, \omega_2) \in \Omega$.  Then $\sigma$ and $D$ are independent and $\P(\sigma=x) = 1/(2x^{2}) \1_{x \in \N}$ for all $x \in [0,\infty]$, $\P(\sigma=\infty) = 1 - \pi^2/12$, and $\P(D = -1) = 1/2 = \P(D=1)$. 

We now set $X \equiv D \sigma^2 \1_{\lc \sigma, \infty\lc}$ and let $\mathbb F$ denote the natural filtration of $X$.  To wit, $X$ is a martingale that at time $\sigma$ jumps to either $\sigma^2$ or $-\sigma^2$ provided that $\sigma$ is finite.  In particular, we have
$$X_\infty = \lim_{t \uparrow \infty} X_t = D \sigma^2 \1_{\{\sigma < \infty\}}.$$
Next, we observe that 
$$\E[|X_\infty|] = \E[\sigma^2 \1_{\{\sigma < \infty\}}] = \sum_{n \in \N} n^2 \frac{1}{2 n^2} = \infty.$$
Hence, $X_\infty \notin L^1$, and $X$ is not a uniformly integrable martingale. 

Now, we let $\tau$ be an arbitrary finite stopping time and set $u = \tau((\infty,-1)) \vee \tau((\infty,1)) \in [0,\infty)$.   Thus, $\{\sigma = \infty\} \subset \{\tau \leq u\} \in \F_u$, which again yields $\{\sigma>u\} \subset \{\tau \leq u\}$ since $\{\sigma>u\}$ is the smallest event in $\F_u$ that contains $\{\sigma =\infty\}$.  Thus, since
\begin{align*}
	\{\tau \wedge \sigma \leq u\} = \{ \sigma \leq u\}  \cup \{\tau  \leq u\}   \supset \{\sigma \leq u\} \cup \{\sigma > u\} = \Omega,
\end{align*}
the stopping time $\tau \wedge \sigma$ is uniformly bounded by $u$ and we obtain that
$X_\tau = X^\sigma_\tau = X_{\tau \wedge \sigma} \in L^1$ and, by the optional sampling theorem again, that $\E[X_\tau] = \E[X_{\tau \wedge \sigma}] =0 = \E[X_0]$.  Hence,  $X$ satisfies  \ref{IV} but not \ref{I}, and therefore neither \ref{II} nor \ref{III}. \qed
\end{ex}
We remark that \citet{Dellacherie1970} discusses the filtration of a similar example.

In order to motivate the arguments in the next section, we now slightly modify Example~\ref{Ex:2} by extending the underlying filtration.

\begin{ex} \label{Ex:3}
We let $(\Omega, \F, \mathbb F, \P)$ be an arbitrary probability space that supports a right-continuous martingale $X = D \sigma^2 \1_{\lc \sigma, \infty\lc}$ with the same distribution as in Example~\ref{Ex:2} and an independent $\F_0$--measurable $(0,1)$--uniformly distributed random variable $U$.   

Our goal is to construct a finite stopping time $\tau$ such that $X_\tau \notin L^1$.  Thus, under this enlarged filtration, the previous example is not a counterexample for the implication \ref{IV} $\Rightarrow$ \ref{I}.  Indeed, we note that $\tau = 1/U$ is a finite stopping time and $|X_\tau| = \sigma^2 \1_{\{\sigma \leq 1/U\}}$. Therefore,
\begin{align*}
	\E[|X_\tau|] =   \E \left[ \sigma^2 \1_{\{\sigma \leq 1/U\}} \right] = \E\left[ \sum_{n=1}^{\lfloor 1/U \rfloor} n^2 \frac{1}{2 n^2} \right] = \frac{1}{2} \E\left[ \left\lfloor \frac{1}{U}\right\rfloor\right] \geq   \E\left[\frac{1}{2 U}\right] - \frac{1}{2}  = \int_0^1 \frac{1}{2y} \dd y - \frac{1}{2} = \infty,
\end{align*}
where $\lfloor\cdot \rfloor$ denotes the Gauss brackets, by independence of $U$ and $X$.  \qed
\end{ex}

Example~\ref{Ex:3} indicates that if ``randomized stopping'' is possible, a non-uniformly integrable martingale $X$ will not satisfy  \ref{IV}. In the next section, we will prove this assertion.

\section{An additonal randomization}  \label{S:random}
In this section, we show that the implication  \ref{IV} $\Rightarrow$ \ref{I} holds if we may randomize stopping times.  More precisely, we shall make the following assumption:
\begin{align}   \tag{$\mathfrak R$}  \label{eq:R}
\begin{split}
	\text{There exists a $(0,1)$--uniformly distributed random variable $U$   }  \\
	\text{and a finite stopping time $\eta$, such that $U$ is $\F_\eta$--measurable.}
	\end{split}
\end{align}

We emphasize that \eqref{eq:R} is an assumption on the underlying filtered probability space $(\Omega, \F, \mathbb F, \P)$, and not on the stochastic process $X$. 
We recall that we already argued that $X$ is a martingale if  \ref{IV} holds.  The conclusion that $X$ is also a uniformly integrable martingale, if
additionally \eqref{eq:R} holds, follows then from the following theorem.

\begin{theorem}	If  \ref{IV} and \eqref{eq:R} hold then so does \ref{I}; to wit, $X$ is then a uniformly integrable martingale.
\end{theorem}
\begin{proof}
Exactly as in the proof of Theorem~\ref{T:1}, we may assume, without loss of generality, that $\mathbb F$ and $X(\omega)$ are right-continuous  for each $\omega \in \Omega$.    Lemmata~\ref{L:1} and \ref{L:2} below then yield that $\liminf_{t \uparrow \infty} |X_t| \in L^1$ and the implication \ref{III} $\Rightarrow$ \ref{I}, proven in  Thereom~\ref{T:1}, yields the assertion.
\end{proof}

\begin{lemma} \label{L:1}
Assume that $\mathbb F$ and $X(\omega)$ are right-continuous  for each $\omega \in \Omega$.
If  $X_\tau \in L^1$ for all finite stopping times $\tau$ and \eqref{eq:R} holds then $\E[\liminf_{t \uparrow \infty} |X_t||\F_\eta] < \infty$. 
\end{lemma}
\begin{proof}
We  define the event $$A = \left\{\E\left[\left. \liminf_{t \uparrow \infty} |X_t|\right|\F_\eta\right] = \infty\right\} \in \F_\eta.$$  We need to argue that $\P(A) = 0$. Towards this end, we assume that $\P(A)>0$ and define the function $g: [0,1] \rightarrow [0,\infty]$ by $t \mapsto 1/\P(A \cap \{U \leq t\})$, where $U$ is the uniformly distributed random variable of \eqref{eq:R}. We note that the function $1/g$ is continuous and nondecreasing and set $t_\infty = \sup \{t \in [0,1]| g(t) = \infty\}$.  Then we have $\P(A \cap \{U \leq t_\infty\}) = 0$ and $\P(A \cap \{U \leq t\}) >0$ for all $t > t_\infty$, which yields that  $\1_A g(U)$ (with $0 \times \infty = 0$) is finite (almost surely). 

We now let $\sigma$ denote the first time $t$ after $\eta$ such that $\E[|X_t||\F_\eta]$ is greater than or equal to $g(U)$ and note that $\sigma$ is a  stopping time. Then, Fatou's lemma yields that $\sigma$ is finite on $A$.  We now set $\tau=\eta \1_{\Omega \setminus A} +\sigma \1_A$, which is again a finite stopping time, and observe
\begin{align*}
	\E[|X_\tau|]  \geq \E\left[\1_A  |X_\sigma| \right] =  \E\left[\1_A  \E[|X_\sigma| | \F_\eta] \right]  \geq \E\left[\1_A  g(U) \right]  \geq \sum_{n \in \N} \P\left(A \cap \left\{\P(A \cap \{U \leq t\})_{t = U} \leq \frac{1}{n}\right\}\right) = \infty.
\end{align*}
Here the last inequality follows from Tonelli's theorem and the last equality follows from the fact that for each $n \geq 1/\P(A)$ the corresponding term in the sum equals $1/n$.   To see this, fix $n \geq 1/\P(A)$, let  $t_n = \sup \{t \in [0,1]| g(t_n) \geq n\}$, and use the fact that $\P(A \cap \{U \leq t_n\}) = 1/n$.   The last display contradicts the assumption and thus yields $\P(A) = 0$.
\end{proof}

\begin{lemma} \label{L:2}
Assume that $\mathbb F$ and $X(\omega)$ are right-continuous  for each $\omega \in \Omega$.
If  $X_\tau \in L^1$ for all finite stopping times $\tau$ and $\E[\liminf_{t \uparrow \infty} |X_t||\F_\eta] < \infty$ holds for some finite stopping time $\eta$ then $\liminf_{t \uparrow \infty} |X_t| \in L^1$. 
\end{lemma}
\begin{proof}
We let $Y = (Y_t)_{t \in [0,\infty)}$ denote the right-continuous modification of the finite-valued conditional expectation process $$\left(\E\left[\left. \liminf_{s \uparrow \infty} |X_s| \right|\F_{\eta \vee t}\right]\right)_{t \in [0,\infty)}.$$  For each $\kappa>0$ the process $(Y_t \1_{\{Y_0 \leq \kappa\}})_{t \in [0,\infty)}$ is a uniformly integrable martingale under its natural filtration and L\'evy's  martingale convergence theorem yields that $\1_{\{Y_0 \leq \kappa\}} \lim_{t \uparrow \infty}Y_t =  \1_{\{Y_0 \leq \kappa\}} \liminf_{s \uparrow \infty} |X_s|$, and thus  $\lim_{t \uparrow \infty}Y_t =  \liminf_{s \uparrow \infty} |X_s|$.

We now let  $\tau$ denote the first time after time $\eta$ such that $|X|$ is greater than or equal to $Y-1$, which is a stopping time. Moreover, $\tau$ is finite since $\liminf_{t \uparrow \infty} |X_t| > \lim_{t \uparrow \infty}Y_t - 1.$   We then obtain
\begin{align*}
 \E\left[\liminf_{t \uparrow \infty} |X_t|\right] =  \E\left[Y_\tau\right] \leq 1 +  \E\left[|X_\tau|\right] < \infty,
\end{align*}
which yields the statement.
\end{proof}

\bibliographystyle{apalike}
\setlength{\bibsep}{1pt}
\bibliography{aa_bib}

\begin{thebibliography}{}

\bibitem[Cherny, 2006]{Cherny_2006}
Cherny, A.~S. (2006).
\newblock Some particular problems of martingale theory.
\newblock In Kabanov, Y., Liptser, R., and Stoyanov, J., editors, {\em From
  Stochastic Calculus to Mathematical Finance: The Shiryaev Festschrift}, pages
  109--124. Springer.

\bibitem[Dellacherie, 1970]{Dellacherie1970}
Dellacherie, C. (1970).
\newblock Un exemple de la th{\'e}orie g{\'e}n{\'e}rale des processus.
\newblock In Meyer, P., editor, {\em S{\'e}minaire de Probabilit{\'e}s, IV
  (Lecture Notes in Mathematics, Volume 124)}, pages 60--70. Springer, Berlin.

\bibitem[Dellacherie and Meyer, 1978]{Dellacherie/Meyer:1978}
Dellacherie, C. and Meyer, P.~A. (1978).
\newblock {\em Probabilities and Potential}.
\newblock North-Holland, Amsterdam.

\bibitem[F\"ollmer, 1972]{F1972}
F\"ollmer, H. (1972).
\newblock The exit measure of a supermartingale.
\newblock {\em Zeitschrift f{\"u}r Wahrscheinlichkeitstheorie und Verwandte
  Gebiete}, 21:154--166.

\bibitem[Hulley, 2009]{Hulley_2009}
Hulley, H. (2009).
\newblock {\em Strict Local Martingales in Continuous Financial Market Models}.
\newblock PhD thesis, University of Sydney, Sydney, Australia.

\bibitem[Stroock and Varadhan, 2006]{SV_multi}
Stroock, D.~W. and Varadhan, S. R.~S. (2006).
\newblock {\em Multidimensional {D}iffusion {P}rocesses}.
\newblock Springer, Berlin.
\newblock Reprint of the 1997 edition.

\end{thebibliography}

\end{document}